\documentclass{amsart}

\usepackage{amsmath, amssymb, amsthm}
\usepackage[mathcal]{eucal}
\usepackage{enumerate}
\usepackage{verbatim, hyperref}
\usepackage{comment}
\usepackage{xcolor}
\usepackage{appendix}

\makeatletter
\renewcommand{\l@subsection}{%
  \@tocline{2}{0pt}{2.5pc}{5pc}{}%
}
\makeatother

\newcommand{\sign}{\sigma}
\newcommand{\lil}{\varepsilon}

\newtheorem{theorem}{Theorem}
\newtheorem*{theorem*}{Theorem}
\newtheorem{lemma}{Lemma}
\newtheorem{question}{Question}
\newtheorem{proposition}{Proposition}

\newtheorem{oldtheorem}{Theorem}

\def\Z{\ensuremath{\mathbb Z}}
\def\C{\ensuremath{\mathbb C}}

\def\T{\ensuremath{\mathbb T}}
\def\D{\ensuremath{\mathbb D}}

\renewcommand{\epsilon}{\varepsilon}
\renewcommand{\phi}{\varphi}
\renewcommand{\bar}{\overline}

\def\Zchi{\ensuremath{\mathbf 1}}

\numberwithin{equation}{section}

\title[Quantitative Baxter's theorem]{Quantitative possibilities for Baxter's theorem \\in nonlinear Fourier analysis}

\author{Michel Alexis}
\address{Clemson University,
O-110 Martin Hall,
220 Parkway Drive,
Clemson, SC, USA, 29634}
\email{micalexis.math@gmail.com}

\author{Gevorg Mnatsakanyan}
\address{Yerevan State University, 1 Alex Manoogian st., 0025, Yerevan, Armenia}
\email{gevorg.mnatsakanyan@ysu.am}

\author{Kristina Oganesyan}
\address{Steklov Mathematical Institute of Russian Academy of Sciences, 8 Gubkina Str., 119333, Moscow, Russia}
\email{kristina.oganesyan@mi-ras.ru}

\date{October 2026}


\begin{document}

\begin{abstract}
 Baxter's classical theorem for orthogonal polynomials on the unit circle establishes that the linear Fourier coefficients of a measure are in $\ell^1$ if and only if the nonlinear coefficients, i.e., the so-called Verblunsky coefficients, are in $\ell^1$. However, this equivalence is purely qualitative. We explore possible formulations of such quantitative theorems with norm and Lipschitz estimates for the $SU(2)$-valued nonlinear Fourier transform, for which the analog of Baxter's theorem has been recently obtained. In particular, the highlight of this paper is a number-theoretic construction for the NLFT which proves some negative results in this direction. We also prove a positive result and discuss the limitations of Baxter's method.
\end{abstract}

\maketitle

\tableofcontents

\section{Introduction}

The nonlinear Fourier transform (NLFT) of  a finitely supported complex sequence $F=(F_n)_{n\in \Z}$ is defined as the matrix-valued function
\begin{equation}\label{def:NLFT}
G_F(z)=G(z) = \prod_{k=-\infty}^{\infty} \frac{1}{\sqrt{1-\sign |F_k|^2}} \begin{pmatrix}
    1 & F_k z^k \\
    \sigma \bar F_k z^{-k} & 1
\end{pmatrix},
\end{equation}
where $ \sign=\pm 1$ and $z$ is on the unit circle $\T$. Depending on whether $\sign$ equals $-1$ or $1$, each of the transfer matrices in \eqref{def:NLFT} and consequently their product $G_F$ belongs to the matrix groups 
\[
 SU(1,1) := \Big\{ \begin{pmatrix}
    \alpha & \beta \\ \overline{\beta} & \overline{\alpha}
\end{pmatrix}:  |\alpha|^2 - |\beta|^2 = 1\Big\} \, , \; SU(2):=\Big\{ \begin{pmatrix}
    \alpha & \beta \\ -\overline{\beta} & \overline{\alpha}
\end{pmatrix}:  |\alpha|^2 + |\beta|^2 = 1\Big\} \, , 
\]
respectively.
If $\sign = 1$, we require $|F_n|<1$ for all $n$, whereas when $\sign = -1$, there is no such restriction. Hence, we will call the corresponding matrices $G_F$ the $SU(1,1)$-NLFT or $SU(2)$-NLFT of $F$, depending on the choice of $\sign$. For $F$ in $\ell^p(\Z)$, $1\leq p\leq 2$, the NLFT can be defined as a limit of the partial products in the appropriate spaces. We refer to \cite{TaoThiele2012,tsai} for the general theory and to \cite{QSP_NLFA} for a recent survey for the $SU(2)$-NLFT. We denote the entries of the matrix $G_F$ by
$$
\begin{pmatrix}
    a_F (z) & b_F (z) \\
    \sign b_F^*(z) & a_F^* (z)
\end{pmatrix} := G_F (z)  \, ,
$$
and will shorthand the matrix by its first row $(a_F,b_F)=(a,b)$. Here and below, $f^*$ denotes the holomorphic reflection of a function $f$ across the unit circle $\T$, i.e. $f^*(z) := \overline{f(\bar z^{-1})}$. One can show \cite[Theorem 3]{QSP_NLFA} that $a^*$ is analytic on $\D$, and hence $a$ is analytic on the disk at infinity $\D^* := \{ z \in \C ~:~ |z| > 1\} \cup \{\infty\}$. Let us also denote by $A_1$ the Wiener algebra of functions on the unit circle with Fourier coefficients in $\ell^1$ with the corresponding norm.

The $SU(1,1)$-NLFT is intimately connected to orthogonal polynomials on the unit circle \cite{simon}, which is an area with a long history, filled with many beautiful and classical results. One such results is Baxter's theorem, which establishes a qualitative $\ell^1$ mapping property. In the language of the NLFT it reads as follows.
\begin{oldtheorem}[A variant of Baxter's theorem \cite{baxter}]
Let $F\in \ell^2(\Z_+)$ and $(a_F,b_F)$ be its $SU(1,1)$-NLFT. Then, $F\in \ell^1(\Z_+)$ if and only if $a_F,b_F \in A_1$.
\end{oldtheorem}

The second author has recently proved an analog for the $SU(2)$ setting.
\begin{oldtheorem}[\cite{SU2Baxter}]\label{oldthm:su2baxter}
Let $F \in \ell^2 (\Z)$ and $(a_F,b_F)$ be the $SU(2)$-NLFT of $F$. If $F\in \ell^1(\Z)$, then $a_F,b_F \in A_1$. Conversely, if $a_F$ is outer on $\D^*$ and $b_F/a_F^* \in A_1$, then $F\in \ell^1(\Z)$.
\end{oldtheorem}
The forward direction in both theorems above is fairly easy and yields nice quantitative norm and Lipschitz bounds. For example, one can get both in $SU(1,1)$ and $SU(2)$ settings \cite[Section 5]{QSP_NLFA},
\begin{equation}\label{eq:norm bd a+b F}
    \|a\|_{A_1}+\|b\|_{A_1} \lesssim e^{\|F\|_{\ell^1}}\, ,
\end{equation}
where we write $f(x) \lesssim g(x)$ if there exists $C>0$ such that $f(x) \leq C g(x)$ for all $x$. Similarly $f(x,y) \lesssim_y g(x,y)$ means that for each $y$ there is $C_y >0$ such that $f(x,y) \leq C_y g(x,y)$ for all $x$. We also write $f\approx g$ whenever we have both $f\gtrsim g$ and $f\lesssim g$.

The interesting part of the proofs of Baxter's theorem is the opposite direction, in particular as they are very qualitative and yield little in the way of norm estimates.

From a different perspective, the $SU(2)$-NLFT is equivalent to generalized quantum signal processing (QSP), a quantum algorithm as in \cite{low2017}. This connection was first established in \cite{QSP_NLFA} if the sequence $F$ is purely imaginary and even, and then in general in \cite{Linlin2025, laneve}. From the QSP perspective, several $A_1 \to \ell^1$ norm and Lipschitz estimates in the vicinity of $0$ were obtained by \cite{Linlin} in the following form.

Henceforth, all NLFTs will be $SU(2)$-NLFTs unless otherwise stated.
\begin{oldtheorem}[A variant of estimates from \cite{Linlin}] \label{oldthm:linlin}
Let $F, F'\in \ell^2(\Z)$ be pure imaginary and even. There exists a constant $r_b<0.5$ such that if $\|b_F\|_{A_1}, \|b_{F'}\|_{A_1}<r_b-\eta$ for some $\eta>0$, then
\begin{equation}\label{eq:b_to_F_Lip}
    \|F-F'\|_{\ell^1} \lesssim_\eta \|b_F-b_{F'}\|_{A_1} \, .    
\end{equation}
\end{oldtheorem}
We draw the reader's attention to the fact that, because of the difference in notation, the constant $r_b$ involved in the statement above is not equal to the corresponding value $r_c$ presented in \cite{Linlin}.
 
The above estimate for a smaller neighborhood of $0$ can be obtained by the straightforward multilinear estimates as demonstrated in \cite[Theorem 5]{QSP_NLFA}, however \cite{Linlin} offers a more sophisticated approach to reach the constant $r_b$. We believe that the same methods can apply to general sequence $F$ (not necessarily pure imaginary and even) maybe shrinking the neighborhood of $0$ a little. Trying to quantify the method of Baxter's proof, the following estimate was obtained in \cite{SU2Baxter}.

\begin{oldtheorem}[\cite{SU2Baxter}]\label{oldthm:quant sqrt 2 baxter}
Assume for $F\in \ell^2(\Z)$, $a_F$ is outer on $\D^*$ and $\|b\|_{A_1}<1/\sqrt{2}-\varepsilon$, then
\begin{equation*}\label{eq:thmD est b/a}
    \|F\|_{\ell^1} \lesssim \frac{1}{\varepsilon} \|b_F/a_F\|_{A_1} \, .
\end{equation*}
\end{oldtheorem}
The fact that $\|b_F/a_F\|_{A_1}$ is finite under the hypothesis of the theorem can be shown using Wiener's lemma and standard Banach algebra properties \cite[Lemma 2]{SU2Baxter}.

The above discussion makes it natural to ask whether quantitative bounds for Baxter's theorem can be proved and in what way. In this note we explore some possible formulations of such quantitative bounds and construct  counterexamples for several of them, with a focus on \eqref{eq:b_to_F_Lip}.
In the remainder of the introduction we will first state our main results and then formulate some open questions motivated by the discussion.

\begin{theorem}\label{prop:counterexamples a b F l1 A1}
For any $\Lambda>0$, there exist $F, H\in \ell^1(\Z)$ such that $a_F, a_H$ are outer on $\D^*$,
\begin{equation}\label{eq:b F counterexample}
     \|F\|_{\ell^1}\geq \Lambda \|b_F\|_{A_1} \text{ with } \|b_F\|_{A_1}\leq 1+\Lambda^{-1}
\end{equation}
and
\begin{equation}\label{eq:a F counterexample}
    \|H\|_{\ell^1} \geq \Lambda \|a_H-1\|_{A_1} \text{ with } \|a_H-1\|_{A_1} \leq \Lambda^{-1} \, .
\end{equation}
\end{theorem}

Next, we show that the Lipschitz $A_1\to \ell^1$ estimate also fails, i.e., \eqref{eq:b_to_F_Lip} cannot hold with uniform constant across all possible potentials $F, F'$. 
\begin{theorem}\label{prop:lipschitz l1 A1}
For any $\Lambda>0$, there exist $F,F'\in \ell^1(\mathbb Z)$ such that $a_F,a_{F'}$ are outer on $\D^*$ and
\begin{equation}\label{eq:Lip_cter_ex_1}
    \|F-F'\|_{\ell^1} > \Lambda \|b_F-b_{F'}\|_{A_1}\, .
\end{equation}
\end{theorem}

We consider outer $a$'s in the above theorems for two reasons. First, in positive results proved so far, including some recent breakthrough results in QSP, outerness of $a$ was mostly assumed, so it's natural to take $a$ outer in our counterexamples to properly delineate the limits of the positive results. Second, the existence of soliton solution immediately implies an estimate like \eqref{eq:b F counterexample}, as one can take a nontrivial $F$ that generates an NLFT pair $(a,0)$ with $a$ an inner function, for example, a finite Blaschke product. We refer to \cite{tsai, FaddeevTakhtajan1987} for soliton constructions.

Whereas Theorem \ref{prop:lipschitz l1 A1} dealt with an $A_1 \to \ell^1$ Lipschitz bound, we now consider a counterexample to an $L^2 \to \ell^2$ Lipschitz bound. The nonlinear Plancherel identity is well known for the NLFT and establishes an $L^2$ norm equivalence for $F$ and $\widehat{\log |a_F|}$. On the other hand, in \cite{QSP5author}, an $L^2\to \ell^\infty$ Lipschitz bound was proved for the case when $a$ is outer and bounded away from $0$. It was of interest \cite{linlinrecentnumerics, lin2026mathematical} whether a $L^2\to \ell^2$ Lipschitz bound was feasible for the NLFT. The next theorem shows that it is not.

\begin{theorem}\label{prop:lipschitz l2}
For any $\Lambda>0$, there exist $F,F'\in \ell^2(\mathbb Z)$ such that $a_F, a_{F'}$  are outer on $\mathbb D^*$  and
$$
 \|F-F'\|_{\ell^2}> \Lambda \|b_F -b_{F'}\|_{L^2}  \, .
$$
\end{theorem}
We point out that counterexamples in Theorems \ref{prop:counterexamples a b F l1 A1}-\ref{prop:lipschitz l2} still hold even if one parametrizes the potential by $\frac{F_j}{\sqrt{1+ |F_j|^2}}$ instead of $F_j$ and takes $\ell^1$ norms with respect to this parametrization.

Let us also state two positive results.
The first one is obtained directly by Theorem 
\ref{oldthm:quant sqrt 2 baxter} applying standard Wiener algebra estimates. It improves the constant $r_b$ from Theorem \ref{oldthm:linlin}.
\begin{theorem}\label{prop:positive}
Let $F\in \ell^2(\Z)$ and assume $a_F$ is outer on  $\D^*$ and $\|b_F\|_{A_1}<1/\sqrt{3}-\varepsilon$, for some $\varepsilon>0$. Then
\begin{equation*}
    \|F\|_{\ell^1} \lesssim \frac{1}{\varepsilon} \|b_F\|_{A_1} \, .
\end{equation*}
\end{theorem}

The next result is obtained by manipulating Baxter's trick.
\begin{theorem}\label{prop:positive with a}
Let $F\in \ell^2(\Z)$ and assume $a_F$ is outer on $\D^*$ and $0<\frac{\|b_F\|_{A_1} }{a_F^*(0)(1- \|b_F\|_{A_1}^2)}<2-\varepsilon$, for some $\varepsilon>0$. Then
\begin{equation}\label{eq:l1_baxter_trick}
    \|F\|_{\ell^1} \lesssim \frac{1}{\varepsilon} \left( \frac{\|a_F\|_{A_1}}{a_F^*(0)}-1 \right) \, .
\end{equation}
\end{theorem}
We note that the sets of pairs $(a,b)$ satisfying the hypotheses of Theorems \ref{prop:positive} and \ref{prop:positive with a} are not subsets of each other. Furthermore, the hypotheses of Theorem \ref{prop:positive with a} apply to many pairs $(a, b)$, for instance, any pair $(a,b)$ satisfying 
\[
\|b\|_{A_1} \leq \frac{2-\epsilon}{\sqrt{1 + (2-\epsilon)^2}} \, ,
\]
and $a$ is the outer function on $\D^*$ with positive mean satisfying $|a|^2 + |b|^2 =1$.

Finally, we state some interesting open questions.

\begin{question}
For a sequence $F$ and its $SU(1,1)$-NLFT $(a_F,b_F)$, do we have
\begin{equation}\label{eq:counterexampe su11 F and b}
    \|F\|_{\ell^1} \lesssim \|b_F\|_{A_1} \, ?
\end{equation}
\end{question}
The proof of \eqref{eq:a F counterexample} of Theorem \ref{prop:counterexamples a b F l1 A1} applies verbatim for the $SU(1,1)$-NLFT and \eqref{eq:a F counterexample} in turn implies that in \eqref{eq:counterexampe su11 F and b} the right-hand side cannot be replaced by $\|a_F\|_{A_1}$.

\begin{question}\label{question:SU2_aplusb}
For a sequence $F$ and its $SU(2)$-NLFT $(a_F,b_F)$ with $a_F$ outer on $\D$, do we have
\begin{equation*}
    \|F\|_{\ell^1} \lesssim \|a_F\|_{A_1}+\|b_F\|_{A_1} \, ?
\end{equation*}
\end{question}
Of course, the point of Question \ref{question:SU2_aplusb} is that our examples of Proposition \ref{prop:counterexamples a b F l1 A1} make only one of $a$ or $b$ small while the other one is large.
Another interesting aspect of our counterexample in \eqref{eq:b F counterexample} is that $\|F\|_{l^\infty}=1$ and $\|b_F\|_{A_1}=1+o(1)$. The first condition is exactly what makes our counterexample unsuitable for the $SU(1,1)$-NLFT, while the second one may be the sharp range of the neighborhood of $0$ for such an estimate.

\begin{question}\label{quest_theshold}
What is the largest value of $\tau$ such that for any $\eta>0$ and any $F\in \ell^2(\Z)$ with $\|b_F\|_{A_1}<\tau-\eta$ and $a_F$ outer on $\D^*$, we have
\begin{equation*}
    \|F\|_{\ell^1} \lesssim_{\eta} 1 \, ?
\end{equation*}
\end{question}
Theorems \ref{prop:counterexamples a b F l1 A1} and \ref{prop:positive} imply $1/\sqrt{3}\leq \tau \leq 1$. The question can also be asked without the outerness assumption but allowing the implicit constant to depend also on the inner part of $a^*$. Indeed, from \cite{tsai}, we know that if $a^*$ is a single Blaschke factor with a zero at $z_0\in \D$ and $b=0$, i.e. $(a,b)$ is a rational soliton data, then the NLFT preimage of $(a,b)$ is one-parameter family of sequences $F(\gamma)$, parametrized by $\gamma\in \C\setminus\{0\}$, given by 
$$F_n(\gamma) = \left( \frac{1}{|z_0|} - |z_0|\right) \frac{\bar \gamma \bar z_0^n}{1+|\gamma|^2 |z_0|^{2n}} \, .
$$
Hence, we can blow up the norm of $F$ letting $z_0$ approach $0$.
\begin{question}
For $\eta>0$, is it true that for any $F\in \ell^2(\Z)$ with $\|F\|_{\ell^\infty}<1-\eta$ and $a_F$ outer on $\D^*$, we have
\begin{equation*}
    \|F\|_{\ell^1} \lesssim_{\eta} \|b_F\|_{A_1} \, ?
\end{equation*}
\end{question}

\subsection*{Acknowledgements}
We acknowledge the use of AI in suggesting Lemma \ref{lem_char} and giving a proof for it, which was checked and modified by us according to our purposes. 

G. Mnatsakanyan was supported by the Higher Education and Science Committee of Republic of Armenia (Research Project No 23RL-1A027).

The work of K. Oganesyan was performed at the Steklov International Mathematical Center and supported by the Ministry of Science and Higher Education of the Russian Federation (agreement no. 075-15-2025-303).

\section{A few useful facts from number theory}\label{sec:nt}
Our counterexamples are based on a construction of weighted sequences of $\pm 1$ that enjoy certain cancellation properties. In order to establish and exploit this cancellation structure, we have to invoke some number-theoretic tools.

Let us first introduce the Legendre symbol, which will define the choice of $\pm 1$ in our construction and eventually take care of the cancellation we aim to achieve inside the multilinear terms of the corresponding functions $a_F$ and $b_F$. For an odd prime number $p$ and an integer $k\neq 0 \mod p$, the Legendre symbol $\big(\frac{k}{p}\big)$ equals $1$ if the equation $x^2=k\mod p$ has an integer solution and $-1$, otherwise. For $k=0\mod p$, we set $\big(\frac{k}{p}\big)=0$. The following lemma will be useful for us.

\begin{lemma}\label{lem:legendre_sum}
    For any odd prime number $p$ and any $\alpha,\beta,\gamma\in\mathbb Z$ with $\alpha\neq 0\mod p$,
    \begin{align*}
        \sum_{m=0}^{p-1}\Big(\frac{\alpha m^2+\beta m+\gamma}{p}\Big)=\begin{cases}
            -\Big(\frac{\alpha}{p}\Big),\quad\text{if}\; \beta^2-4\alpha\gamma \neq 0\mod p,\\
            (p-1)\Big(\frac{\alpha}{p}\Big),\quad\text{otherwise}. 
        \end{cases}
    \end{align*}
\end{lemma}

\begin{proof}
Using the multiplicativity of the Legendre symbol and the obvious fact that $\big(\frac{4a^2}{p}\big)=1$, we obtain
\begin{align*}
        \sum_{m=0}^{p-1}\Big(\frac{\alpha m^2+\beta m+\gamma}{p}\Big)&=\sum_{m=0}^{p-1}\Big(\frac{4\alpha^2(\alpha m^2+\beta m+\gamma)}{p}\Big)\\
        &=\sum_{m=0}^{p-1}\Big(\frac{\alpha}{p}\Big)\Big(\frac{(2\alpha m+\beta)^2-(\beta^2-4\alpha\gamma)}{p}\Big)\\
        &=\Big(\frac{\alpha}{p}\Big)\sum_{m=0}^{p-1}\Big(\frac{m^2-(\beta^2-4\alpha\gamma)}{p}\Big).
    \end{align*}
    If $\beta^2-4\alpha\gamma=0\mod p$, the statement follows immediately. Otherwise, denote $r:=\beta^2-4\alpha\gamma\neq 0\mod p$ and observe that for a fixed $m$, the number of solutions $x \in [0,p-1]$ to $x^2=m^2-r \mod p$ is $1+\big(\frac{m^2-r}{p}\big)$. On the other hand, the number of pairs $(m,x) \in [0, p-1]^2$ satisfying $m^2-x^2=(m-x)(m+x)=r$ is $p-1$, whence 
    \begin{align*}
        \sum_{m=0}^{p-1}\Big(\frac{m^2-r}{p}\Big)=\sum_{m=0}^{p-1}\Big(\Big(\frac{m^2-r}{p}\Big)+1\Big)-p=(p-1)-p=-1,
    \end{align*}
and the proof is complete.
\end{proof}

We recall that, for a finite abelian group $\mathcal G$, a (multiplicative) character is a homomorphism $\chi:\mathcal G\to \mathbb C^*$, while being primitive means, roughly speaking, that it cannot be appropriately factorized (we refer to \cite[\S3.3]{iwaniec_kowalski} for the precise definition and further discussion of primitive characters). Considered on the multiplicative group $\mathbb Z_p^*$ of nonzero residues modulo $p$, the Legendre symbol is a particular case of characters. Moreover, for an odd prime $p$, the only primitive character modulo $p$ is precisely the Legendre symbol (see \cite[p. 46]{iwaniec_kowalski}). 

In what follows we will need the following lemma.
\begin{lemma}\label{lem_char}
    For any primitive character $\chi\mod p,\;p\geq 3$, any positive integer $M$ and $m$, we have
\begin{align*}
\sup_{\alpha\in\mathbb{R}}\Big|\sum_{n=M+1}^{M+m}\chi(n)e(\alpha n)\Big|\lesssim \sqrt{p}\log p+\frac{m}{\sqrt{p}},
\end{align*}
where $e(x):=e^{2\pi i x}$.
\end{lemma}

\begin{proof}
    Writing the Fourier expansion (see \cite[(3.12)]{iwaniec_kowalski}) for $\chi$, we get
    \begin{align*}
        \chi(n)=\frac{1}{\tau(\overline{\chi})}\sum_{t\mod p}\overline{\chi}(t)e\Big(\frac{tn}{p}\Big),
    \end{align*}
       where $\tau(\overline{\chi})=\sum_{k=0}^{p-1}\overline{\chi}(k)e(k/p)$ is the Gauss sum for $\overline{\chi}$ and $|\tau(\overline{\chi)}|=\sqrt{p}$ (see \cite[(3.14)]{iwaniec_kowalski}). Thus, denoting for a moment by $\|s\|$ the distance from the number $s\in\mathbb R$ to the nearest integer, we arrive at
    \begin{align*}
        \Big|\sum_{n=M+1}^{M+m}\chi(n)e(\alpha n)\Big|&=\frac{1}{\sqrt{p}}\Big|\sum_{n=M+1}^{M+m}\sum_{t\mod p}\overline{\chi}(t)e\Big(\frac{tn}{p}+\alpha n\Big)\Big|\\
        &\leq \frac{1}{\sqrt{p}}\sum_{t\mod p}\min \Big(m, \frac{1}{2\|\alpha+t/p\|}\Big)\\
        &\lesssim \sqrt{p}\log p+\frac{m}{\sqrt{p}},
    \end{align*}
    where we used the fact that $|\chi(a)|=1$ and the usual estimate for the geometric progressions of the form $\sum_{n=r+1}^{R}e(sn)$.
\end{proof}

\section{Construction of the sign pattern and its properties}

It is well known that if $F$ has small $\ell^1$-norm, then $a_F=1 + O(\|F\|_{\ell^1}^2)$ and $b_F=\widehat{F} + O(\|F\|_{\ell^1}^3)$, where $\widehat{F}$ is the linear Fourier transform of the sequence. The main idea for providing the announced counterexamples is the construction of an $F$ which has large $\ell^1$-norm and yet the same approximations for $a$ and $b$ continue, thanks to cancellations in the higher-order  terms of $a$ and $b$. 

Let $p$ be an arbitrary large prime number, let $s>0$ be an integer, and let $\lil >0$ be small. We set

\begin{equation*}
    H_k=H_k(p,s,\lil ):=\begin{cases}
        \lil \Big(\frac{k}{p}\Big),\quad \text{if}\;p\nmid k,\;0\leq k\leq sp-1,\\
        0,\quad\text{otherwise} \, .
    \end{cases}
\end{equation*}

\begin{proposition}\label{prop:main construction}
We have
\begin{align*}
    \|b_H - \widehat{H} \|_{A_1} &\lesssim \lil^3 (sp)^2 + \lil^2 (p+s) sp e^{\lil  s p},\\
    \|a_H- 1\|_{A_1} &\lesssim \lil^2 (p+s) sp  e^{\lil s p}.
\end{align*}
\end{proposition}

For the proof of Proposition \ref{prop:main construction}, we make use of the multilinear expansion for $(a_H,b_H)$ (see e.g. \cite[Section 5]{QSP_NLFA}). Namely, since $H$ is a real summable sequence, then its NLFT data $(a_H,b_H)$ is given by the following absolutely convergent series 
\begin{equation}\label{multi_bG_0}
b_H = C(H) \sum_{n=0}^{\infty} (-1)^n T_{2n+1} (H) \, , \quad a_H = C(H) \sum\limits_{n=0} ^{\infty} (-1)^n T_{2n} (H) \, ,
\end{equation}
where
$$
T_{m} (H) = \sum_{j_1<\dots < j_{m}} \prod_{k=1}^{m} z^{(-1)^{k+1} j_k} \prod_{k=1}^m H_{j_k} \, 
$$
and 
\begin{align}\label{c(f)_0}
C(H) = (1+\lil^2)^{\frac{sp}{2}} \, .
\end{align}
While $T_0 (H)$ and $T_1 (H)$  appear on the left sides of Theorem \ref{prop:main construction}, the multilinear terms $(T_m)_{m \geq 2}$ of $a_H$ and $b_H$ are rendered small by crucial cancellation which in turn is obtained through summation properties of the Legendre symbol (Lemma \ref{lem:legendre_sum}). In this vein, in the upcoming lemma we estimate a sum of products of $H$'s using our choice of signs.

For positive integers $J>r$, let
\begin{equation*}
    R(J,r) := \sum_{j=0}^{J-r} H_{j} H_{j+r} \, .
\end{equation*}

\begin{lemma}\label{lem:main ce estimate - quadratic residues}
We have
\begin{align*}
    |R(J,r)|\leq \begin{cases}
        \lil^2 \Big(p+2\Big\lceil \frac{J}{p}\Big\rceil \Big), \quad \text{if}\;p\nmid r,\\
     \lil^2 J,\quad\text{otherwise}.
    \end{cases}
\end{align*}
\end{lemma}
\begin{proof}
For simplicity, denote $t:=\lceil \frac{J}{p}\rceil$.

First, assume $p \nmid r$. By $p$-periodicity of $H_j$ and then multiplicativity of the Legendre symbol, we have 
\begin{align}
    \lil^{-2} R(J,r) &\leq p+\lil^{-2}\Big|\sum_{j=0}^{p\lfloor (J-r)/p \rfloor -1} H_{j}H_{j+r}\Big|  \leq p+t \lil^{-2} \Big| \sum_{j=0}^{p-1} H_{j} H_{j+r} \Big| \nonumber\\
    &\lesssim p+t \Big| \sum_{1\leq j<p,\;p\nmid j+r}\Big(\frac{j}{p}\Big) \Big( \frac{j+r}{p} \Big) \Big| \nonumber\\
    &= p+t \Big| \sum_{1\leq j<p,\;p\nmid j+r} \Big( \frac{r^2 m(j) (m(j)+1)}{p}\Big)\Big|\label{eq:m_sub} \, ,
    \end{align}
    where we set $m(j):= j r^{-1} \mod p$, which is well-defined since $p \nmid j$ and $p \nmid j+r$. Reindexing the sum and then applying Lemma \ref{lem:legendre_sum} to the polynomial $x(x+1)$, whose discriminant is nonvanishing $\mod p$, we obtain that the right-hand side of \eqref{eq:m_sub} equals 
    \begin{align*}
    p+t\Big|\sum_{1\leq m<p-1}\Big(\frac{m(m+1)}{p}\Big)\Big| &\leq p+t+t\Big|\sum_{1\leq m<p}\Big(\frac{m(m+1)}{p}\Big)\Big|\\
    &=p+t+t\Big|-\Big(\frac{1}{p}\Big)\Big|=p+2t \, .
    \end{align*}

For the case of $p\mid r$, we use the crude estimate 
\begin{equation*}
    \lil^{-2}R(J,r)=\lil^{-2}\sum_{j=0}^{J-r} |H_{j} H_{j+r}|  = \sum_{j=0}^{J-r} 1 \leq J. \qedhere
\end{equation*}
\end{proof}

\begin{proof}[Proof of Proposition \ref{prop:main construction}.]
Recalling the multilinear expansions \eqref{multi_bG_0} of  $a_H$ and $b_H$, we first note that \eqref{c(f)_0} yields 
\begin{align}\label{c(f)}
C(H) = 1 + O(\lil^2 sp).
\end{align}
As for estimating the multilinear terms, for each $m \geq 2$, we write
\begin{align*}
    T_{m} (H) = \sum_{0\leq j_3<\dots < j_{m}\leq sp-1 } \prod_{k=3}^{m} z^{-(-1)^k j_k}\prod_{3\leq k \leq m } H_{j_k} \sum_{r=1}^{j_3-1} R(j_3,r) z^{-r} \, .
\end{align*}
Using Lemma \ref{lem:main ce estimate - quadratic residues}, we estimate
\begin{align*}
\| T_{m}(H)\|_{A_1} &\leq \sum_{0\leq j_3<\dots < j_{m}\leq sp-1 } \lil^{m-2} \sum_{r=1}^{j_3-1} |R(j_3,r)|\\
&\lesssim \sum_{0\leq j_3<\dots < j_{m}\leq sp } \lil^{m-2} \lil^2 (p+s)j_3 \leq \lil^{m} \frac{(ps)^{m-1}}{(m-2)!} (p+s) \, .
\end{align*}
Using that $T_1 (H) = \widehat{H}$, we then have by \eqref{multi_bG_0} and \eqref{c(f)}
\begin{align*}
\|b_H- \widehat{H}\|_{A_1} &\leq |C(H) -1| \|T_1 (H)\|_{A_1} + \sum\limits_{n \geq 1} \|T_{2n+1} (H)\|_{A_1}\\
&\lesssim  \lil^3 (ps)^2 + \lil^2 (p+s) ps  \sum\limits_{n \geq 1} \frac{(\lil ps)^{2n-1}}{(2n-1)!} \\
&\leq \lil^3 (ps)^2 + \lil^2 (p+s)ps e^{\lil ps} \, .
\end{align*}
Similarly,
\begin{align*}
\|a_H-1\|_{A_1} &\leq |C(H) -1| + \sum\limits_{n \geq 1} \|T_{2n} (H)\|_{A_1} \\
&\lesssim \lil^2 ps +  \lil^2 (p+s)ps  \sum\limits_{n \geq 1} \frac{(\lil ps)^{2n}}{(2n)!} \lesssim  \lil ^2 (p+s)ps e^{\lil ps} \, . 
\end{align*}
\end{proof}

\section{Setting up the potentials}\label{sec:defn_potentials}

Here we set up the potentials that will be used in the proofs of the main theorems and derive the main bound for proving outerness of $a$ in these theorems.

For $0<\eta<1$, define the sequence $F'=F'(\eta)$ by
$$
F'_k := \begin{cases}
    1-\eta, \text{ if } k = -1 \text{ or } k=-p-1 ,\\
    0,\text{ otherwise},
\end{cases}
$$
and $F=F (p,\eta,\lil) := F'(\eta) + H(p,p,\lil)$, that is,
\begin{align*}
F_k: = \begin{cases}
    1-\eta,\text{ if } k= -1 \text{ or } k=-p-1,  \\
    H_k (p,p,\lil), \text{ otherwise}.
\end{cases}
\end{align*}
In the next section, we will choose parameters $\epsilon,p,\eta$ in such a way that $F$ and $F'$ will give the potentials satisfying the statements of the main theorems.

We have
$$
\begin{pmatrix}
    a_{F'} & b_{F'} \\
    -b_{F'}^* & a_{F'}
\end{pmatrix} = \frac{1}{1+(1-\eta)^2} \begin{pmatrix}
    1- (1-\eta)^2z^{-p} & (1-\eta)(z^{-1} + z^{-p-1}) \\
    -(1-\eta)(z+z^{p+1}) & 1-(1-\eta)^2z^p
\end{pmatrix}
$$
and
$$
\begin{pmatrix}
    a_F & b_F \\
    -b_F^* & a_F
\end{pmatrix} = \begin{pmatrix}
    a_{F'} & b_{F'} \\
    -b_{F'}^* & a_{F'}
\end{pmatrix} \begin{pmatrix}
    a_H & b_H \\
    -b_H^* & a_H^*
\end{pmatrix} \, .
$$
Hence,
\begin{equation}\label{eq:prod F' and H}
   \begin{cases}
       a_F= a_{F'} a_H -b_{F'}b^*_H \\
       b_F = a_{F'} b_H + b_{F'} a^*_H \, .
   \end{cases}
\end{equation}
The following lemma will ensure outerness of the function $a_F $ for appropriate choices of parameters.
\begin{lemma}\label{lem:outerness inequality}
For any $z\in \D^*$,
\begin{equation*}
    |a_F (z)| - \eta \gtrsim - (\epsilon p^{\frac{3}{2}} + \epsilon^3 p^4 + \epsilon^2 p^3 e^{\epsilon p^2} ) \, .
\end{equation*}
\end{lemma}

\begin{proof}
For $z\in \T$, by \eqref{eq:prod F' and H} we can estimate
\begin{align*}
|a_F (z) - a_{F'}(z)| &\leq | a_{F'}(z) (a_H(z) - 1) | + |b_{F'}(z) b^*_H (z)|\\
&\leq |a_H(z)-1| + |\widehat{H}(z)| + |\widehat{H}(z)- b_H(z)|\\
&\leq \| a_H-1\|_{A_1}+\|\widehat{H}(z)- b_H(z)\|_{A_1} + \| \widehat{H}\|_{L^\infty(\T)}\, ,
\end{align*}
using Proposition \ref{prop:main construction}, we continue
\begin{align*}
&\lesssim \epsilon \left\| \sum_{k=0}^{p^2-1} \Big( \frac{k}{p} \Big) z^{k+1}\right\|_{L^\infty(\T)} + \epsilon^3 p^4 + 2 \epsilon^2 p^3 e^{\epsilon p^2}  \\
&= \epsilon \Big|\sum_{k=0}^{p^2-1} \Big( \frac{k}{p} \Big) z^k \Big| +  \epsilon^3 p^4 + 2\epsilon^2 p^3 e^{\epsilon p^2} \, .
\end{align*}
Finally, by lemma \ref{lem_char}, we have
$$
\lesssim \epsilon p^{\frac{3}{2}} + \epsilon^3 p^4 + \epsilon^2 p^3 e^{\epsilon p^2} \, ,
$$
uniformly for all $z\in \T$. By the maximum principle, the same bound also holds for all $z\in \D^*$. Thus, for all $z\in \D^*$ and some absolute constant $c$, we conclude
$$
|a_F(z)| \geq |a_{F'}(z)| - c \Big( \epsilon p^{\frac{3}{2}} + \epsilon^3 p^4 + \epsilon^2 p^3 e^{\epsilon p^2} \Big) \, .
$$
As $|a_{F'}(z)| \geq \frac{\eta (2-\eta)}{2-2\eta + \eta^2}\geq \eta$, we finish the estimate
$$
\geq \eta - c \Big( \epsilon p^{\frac{3}{2}} + \epsilon^3 p^4 + \epsilon^2 p^3 e^{\epsilon p^2} \Big) \, . 
$$

\end{proof}

\section{Negative results}

\subsection{Proof of Theorem \ref{prop:lipschitz l1 A1}}

Fix $\Lambda >0$. Let $M\geq 1$ be a large parameter and $\delta \in (0,1)$ a small parameter, both of which will be determined later. Define $\lil = \lil (M, \delta)$ and $\eta = \eta (M,\delta)$ by
\begin{align*}
    \lil:=\frac{M^{-3}\delta^2}{e^{2M}}\,,\quad \eta:= \delta^{1/4} \, .
\end{align*}
Let $p = p(M, \delta)$  be a prime satisfying
$$
\frac{M^2e^M}{2\delta}\leq p < \frac{M^2e^M}{\delta} \, ;
$$
 which exists by Bertrand's postulate.
 Taking $F,F', H$ as in the previous subsections, we get
\begin{align*}
    \|F-F'\|_{\ell^1} = \|H\|_{\ell^1} = \lil p(p-1) \approx  M.
\end{align*}
Meanwhile, 
\begin{align}\label{b_H_2}
\| b_F - b_{F'} \|_{A_1} &\leq \frac{1}{2} \|b_H - b_H(1-\eta)^2 z^{-p} \|_{A_1} + \|a_H ^* - 1\|_{A_1} \, .
\end{align}
By Proposition \ref{prop:main construction}, the last term in \eqref{b_H_2} may be estimated as 
\begin{equation}\label{eq:bd_a_1}
\|a_H^*-1\|_{A_1} \lesssim \epsilon^2 p^3  e^{\epsilon p^2} \leq \delta,
\end{equation}
while for the penultimate one, we obtain by the triangle inequality 
\begin{align}\label{eq:comp_shifts}
\|b_H-b_H (1-\eta)^2 z^{-p}\|_{A_1}& \leq \|b_H-b_H z^{-p}\|_{A_1}+ 2\eta \|b_H\|_{A_1} \nonumber\\
&\leq \|\widehat{H} - z^{-p} \widehat{H}\|_{A_1} +2\|b_H -\widehat{H}\|_{A_1} + 2 \eta \|\widehat{H}\|_{A_1} \, .
\end{align}
We can bound the last term in \eqref{eq:comp_shifts} directly from the definition of $H$, and the before-last one, using Proposition \ref{prop:main construction}. As for the first term in \eqref{eq:comp_shifts}, we note that it is the norm of a difference of 
$\widehat{H}$ with its Fourier shift $z^{-p} \widehat{H}$. Because $H$ is $p$-periodic with support of size $p^2$ and height $\lil$, then this difference has Fourier support on at most $2p$ frequencies, and so the $A_1$ norm is at most $2 \lil p $. Combining these three estimates together yields  
\begin{align}
 \|b_H-b_H (1-\eta)^2 z^{-p}\|_{A_1}&\lesssim \varepsilon p + \varepsilon^3 p^4 + \varepsilon^2 p^3 e^{\varepsilon p^2} + \eta\varepsilon p^2 \nonumber \\
&\leq \delta M^{-1}e^{-M} + \delta^2 M^{-1} e^{-2M} +\delta + \eta M \nonumber\\
&\leq 2\delta + \delta^{\frac{1}{4}} M \label{eq:bd_b_shift}\, .
\end{align}
Plugging the bounds \eqref{eq:bd_a_1} and \eqref{eq:bd_b_shift} back into \eqref{b_H_2}, we get
\begin{align*}
\| b_F - b_{F'} \|_{A_1}  \leq C (\delta + \delta^{\frac{1}{4}} M),
\end{align*}
for some absolute constant $C$. By taking $M$ large and then taking $\delta$ sufficiently small, we obtain \eqref{eq:Lip_cter_ex_1}

We are left with verifying outerness of $a_{F}$ and $a_{F'}$. Since both are Laurent polynomials, it suffices to verify neither has zeros in $\D^*$. This is obvious for $a_{F'}$. To ensure outerness of $a_F$, we use Lemma \ref{lem:outerness inequality} and obtain
\begin{align*}
a_F (z) - \delta^{\frac{1}{4}} &\gtrsim - (\epsilon p^{\frac{3}{2}} + \epsilon^3 p^4 + \epsilon^2 p^3 e^{\epsilon p^2} )\\
&\geq - \left( \delta^{\frac{1}{2}} e^{-\frac{M}{2}} +\delta^{2} M^{-1}e^{-2M} + \delta  \right) \gtrsim -\delta^{\frac{1}{2}} \, .
\end{align*}

Making $\delta$ small enough completes the proof. \hfill\qed

\subsection{Proof of Theorem \ref{prop:counterexamples a b F l1 A1}}
To prove Theorem \ref{prop:counterexamples a b F l1 A1}, we take $F$, $F'$ and $H$ as above with parameters $\epsilon (M, \delta), \;\eta(M, \delta),\; p(M, \delta),$ all depending on $M \geq 1$ and $\delta \in (0, 1)$ which will be chosen later. We focus first on the first half of Theorem \ref{prop:counterexamples a b F l1 A1} involving \eqref{eq:b F counterexample}. The function $a_F$ is outer by the same reasoning as for the proof of Theorem \ref{prop:lipschitz l1 A1} provided that $\delta$ is small enough. 

For the norm on the right-hand side of \eqref{eq:b F counterexample}, we have $\|F\|_{\ell^1 }  \approx M $. As for the left-hand side, we have
\[
\|b_F\|_{A_1} \leq \|b_F - b_{F'} \|_{A_1} + \|b_{F'}\|_{A_1} \lesssim \delta + \delta^{\frac 1 4} M + (1-\eta) \lesssim 1  + \delta^{\frac 1 4} M \, .
\]
Taking $M$ large and then $\delta$ small yields \eqref{eq:b F counterexample}.

For the second half of Theorem \ref{prop:counterexamples a b F l1 A1}, we consider $H$. Recall that $\|H\|_{\ell^1} \approx M$ and by Proposition \ref{prop:main construction}, $\|a_H -1 \|_{A_1} \lesssim \delta$. Taking $M$ sufficiently large yields \eqref{eq:a F counterexample}.
As for verifying that $a_H$ is outer, for any $z \in \D^*$, we estimate using the maximum principle and then  \eqref{eq:bd_a_1}
\begin{align*}
    1-|a_H(z)|  \leq \|a_H-1\|_{L^{\infty}} \leq \|a_H-1\|_{A_1} \lesssim\delta,
\end{align*}
whence, provided that $\delta$ is small enough, $a_H$ has no zeros in $\D^*$ and must therefore be outer. This completes the proof of Theorem \ref{prop:counterexamples a b F l1 A1}. \hfill\qed

\subsection{Proof of Theorem \ref{prop:lipschitz l2}}

Here we choose large prime $p$ and put $\epsilon=p^{-2}$. Then,
\begin{align*}
    \|b_F - b_{F'}\|_{L^2} &\leq \|a_H^*-1\|_{L^2} + \| (1-(1-\eta)^2 z^{-p} ) b_H \|_{L^2}\\
&\leq \|a_H^*-1\|_{A_1} + \| b_H-\widehat{H}\|_{A_1} + \| (1-(1-\eta)^2z^{-p} ) \widehat{H} \|_{L^2}\\
&\leq \|a_H^*-1\|_{A_1} + \| b_H-\widehat{H}\|_{A_1} + \| \widehat{H}-z^{-p} \widehat{H} \|_{L^2} + \eta \|\hat{H}\|_{L^2}.
\end{align*}
By Proposition \ref{prop:main construction} and Lemma \ref{lem_char}, we estimate this as 
$$
\lesssim \epsilon^2 p^3 e^{\epsilon p^2} + \epsilon^3 p^4 + \epsilon p^{\frac{1}{2}} + \eta \epsilon p^{\frac{3}{2}}  \lesssim \frac{1}{p} + \eta \frac{1}{\sqrt{p}} \lesssim \frac{1}{\sqrt{p}} \, .
$$
All in all, we get
$$
\|b_F-b_{F'}\|_{L^2} \lesssim \frac{1}{\sqrt{p}} \, .
$$
To ensure outerness of $a$, we need by Lemma \ref{lem:outerness inequality}, 
$$
\frac{1}{\sqrt{p}} < c \eta \, 
$$
with some large absolute constant $c$. So, we fix $\eta$ and choose $p$ large enough to have the last inequality. \hfill\qed

\section{Positive results}
In this section, we first prove Theorem \ref{prop:positive}, then we present the general scheme of the proof of Baxter's theorem, explain Baxter's trick, and apply it to prove Theorem \ref{prop:positive with a}.

\subsection{Proof of Theorem \ref{prop:positive}}
We deduce the theorem from Theorem \ref{oldthm:quant sqrt 2 baxter} by a standard Wiener algebra norm estimate for the inverse. The number $1/\sqrt{3}<1/\sqrt{2}$ appears naturally from the sharp range of norm control for the inverse. We refer to \cite{ElFallahNikolskiZarrabi1998} for such norm control estimates and discussion.

Below we will omit the subscript of $a_F$ and just write $a$. We want to estimate $\|\log |a|\|_{A_1}$. Let us write
$$
|a|^2 = \widehat{|a|^2}(0)\left( 1 + \frac{|a|^2-\widehat{|a|^2}(0)}{\widehat{|a|^2}(0)} \right) \, .
$$
Denote $h:= \frac{|a|^2-\widehat{|a|^2}(0)}{\widehat{|a|^2}(0)}$. We can estimate
$$
\| |a|^2-\widehat{|a|^2}(0)\|_{A_1} = \| |a|^2 \|_{A_1} - \widehat{|a|^2}(0) < 1 + \| |b|^2 \|_{A_1} - \widehat{|a|^2}(0) \leq \frac{4}{3} -\widehat{|a|^2}(0)\, .
$$
Hence, to ensure $\|h\|_{A_1}<1$ it is sufficient to have
$$
\frac{\frac{4}{3}-\widehat{|a|^2}(0)}{\widehat{|a|^2}(0)} < 1 \, ,
$$
which is equivalent to
$$
\widehat{|a|^2}(0) > \frac{2}{3} \, .
$$
The last inequality is true as
$$\widehat{|a|^2}(0) = \int_\T |a|^2 = \int_\T 1-|b|^2 > 1-\frac{1}{3} + \varepsilon = \frac{2}{3} + \varepsilon \, .$$
We can write
$$
\log |a|^2 = \log \widehat{|a|^2}(0) +\log (1+h) = \log \widehat{|a|^2}(0)  + \sum_{k=1}^\infty (-1)^{k-1} \frac{h^k}{k}\, .
$$
Taking the norms, we estimate
$$
\|\log|a|^2\|_{A_1} \leq |\log \widehat{|a|^2}(0)| + \sum_{k=1}^\infty \frac{\|h\|_{A_1}^k}{k} \leq |\log \widehat{|a|^2}(0)| - \log (1-\|h\|_{A_1})
$$
$$
\leq -\log \widehat{|a|^2}(0) - \log \left( 1 - \frac{\frac{4}{3}-\widehat{|a|^2}(0)}{\widehat{|a|^2}(0)} \right) = -\log \left( 2 \widehat{|a|^2}(0) - \frac{4}{3} \right) \leq -\log \varepsilon \, .
$$
Finally, as $a^{-1}$ is also outer, we can write 
$$
\|a^{-1}\|_{A_1} = \| \exp (-\log |a| - i H(\log |a|) ) \|_{A_1} \, ,
$$
where $H$ is the circular Hilbert transform. Using the fact that $H$ has operator norm $\leq 1$ on $A_1$, we estimate this last term by 
$$\leq \exp (2\|\log |a|\|_{A_1}) \leq \exp (-\log \varepsilon) = 1/\varepsilon \, ,
$$
and
$$\|b/a^*\|_{A_1}\leq \|a^{-1}\|_{A_1} \|b\|_{A_1}\leq \varepsilon^{-1} \|b\|_{A_1}\, .$$
\hfill\qed

\subsection{Preparing to prove Theorem \ref{prop:positive with a}: Baxter's trick}

Let us briefly describe the proof of Baxter's theorem and introduce the necessary notation. For a sequence $F\in \ell^2(\Z)$, let $(a_n,b_n)$ be the NLFT of the truncated sequence $F\Zchi_{(-\infty,n]}$ and let
$$
A_n := a_n/a_n^*(0)\, , \quad B_n := b_n/a_n^*(0)
$$
be renormalized versions of $a_n$ an $b_n$. Similarly, put $A:= a_F/a^*_F(0)$ and $B:= b_F/a^*_F(0)$. Clearly, the definition \eqref{def:NLFT} of the NLFT implies a matrix recursive relation for $a_n$ and $b_n$. Renormalized for $A_n,B_n$ it takes the following form:
\begin{equation}\label{eq:normalized recursion}
    \begin{cases} A_{n+1} = A_n - \bar F_{n+1} z^{-n-1} B_n \\
    B_{n+1} = B_n + F_{n+1}z^{n+1}A_n \, .
    \end{cases}
\end{equation}
The proof of the difficult part of Baxter's theorem consists of two parts. The first part is to show that $A_n\to A$ and $B_n\to B$ in the Wiener algebra $A_1$ as $n\to \infty$. The second part is the so-called Baxter's trick. We utilize the second equation of \eqref{eq:normalized recursion} to obtain a telescoping sum. Multiplying both sides by $A^{-1}$, rewrite
$$
-B_n/A = - B_{n+1}/A + F_{n+1} z^{n+1} A_n/A.
$$
As $A$ has Fourier support on nonpositive frequencies, the left-hand side above has Fourier support on $(-\infty,n]$, whereas $-B_{n+1}/A$ and $F_{n+1} z^{n+1} A_n/A$ have Fourier supports on $(-\infty, n+1]$ with coefficient of the latter at frequency $n+1$ being $F_{n+1} A_n^*(0)/A^*(0) = F_{n+1}$. Thus, the $(n+1)$th frequencies of the two terms should cancel. Taking the $A_1$ norms and a triangle inequality, this implies
\begin{equation}\label{eq:baxtertrick telescoping}
    \|B_n/A\|_{A_1} \leq \|B_{n+1}/A\|_{A_1} + |F_{n+1}| \|A_{n}/A\|_{A_1} - 2|F_{n+1}| \, .
\end{equation}
Note that under the assumption of Theorem \ref{oldthm:su2baxter} it is not difficult to see by some Banach algebra properties that $A^{-1}\in A_1$ \cite[Lemma 2]{SU2Baxter}, so the norms in the above display are finite. Furthermore, as $A_n\to A$ by the first part, we have $A_n/A \to 1$ in $A_1$ as $n\to \infty$. This implies that there exists $n_0$ such that for all $n\geq n_0$, we have $\|A_n/A\|_{A_1} < 1.5$, hence adding \eqref{eq:baxtertrick telescoping} for $n=n_0,n_0+1,\dots, N$ we get
$$
\|B_{n_0}/A\|_{A_1} \leq \|B_{N+1}/A\|_{A_1} + \sum_{n=n_0}^N |F_{n+1}| (\|A_n/A\|_{A_1}-2) \, .
$$
Taking the limit as $N\to\infty$, 
$$
0.5 \sum_{n=n_0}^{\infty} |F_{n+1}| \leq \|B/A\|_{A_1} = \|b_F/a_F\|_{A_1} \, .
$$
Applying the same argument on the reflected potential gives the same estimate from the left.

The problem in the above argument is the lack of control on $n_0$ and on the rate of convergence of $a_n$ and $b_n$ to $a$ and $b$.

\subsection{Proof of Theorem \ref{prop:positive with a}}

Let us rewrite the first equation of \eqref{eq:normalized recursion} as
$$
A_n = A_{n+1} + \bar F_{n+1} B_n z^{-n-1} \, ,
$$
and apply the same telescoping Baxter's trick for $n=0,\dots, N$. We get
$$
\|A_0\|_{A_1} \leq \|A_{N+1}\|_{A_1} + \sum_{n=0}^N |F_n| (\|B_n\|_{A_1}-2) \, .
$$
If for some $\epsilon>0$ and all $n\in \Z$, we have $\|B_n\|_{A_1}<2-\epsilon$, then
$$
\sum_{n=0}^\infty |F_n| \leq \frac{1}{\epsilon} (\|A\|_{A_1}-1) = \frac{1}{\epsilon} (\frac{\|a\|_{A_1}}{a^* (0)} - 1) \, .
$$
Applying the same reasoning to the pair $(a^*(z^{-1}), b(z^{-1}))$, which is the NLFT of $(F_{-n} \Zchi_{\{n \leq 0\}} )_{n \in \Z}$ by \cite[(3.11)]{QSP_NLFA}, and using the fact that $a$ being outer implies $a^* (z^{-1})$ is as well with the same $A_1$ norm and value at the origin, we obtain a similar estimate on $\sum\limits_{n=-\infty}^{0} |F_n|$ and hence \eqref{eq:l1_baxter_trick} holds. Let us now show that the hypothesis of the theorem implies $\|B_n\|_{A_1}<2-\epsilon$. By the Riemann-Hilbert equations \cite[(2.12)]{SU2Baxter}, we have
$$
B_n = \frac{1}{a_n^*(0)} b_n = \frac{a^*(0)}{a_n^*(0)} P_{\leq n } \left( B \frac{A_n^*}{A^*} \right) \, ,
$$
where $P_{\leq n}$ is the Fourier projection operator on frequencies $\leq n$.
Then, 
$$
\|B_n\|_{A_1} \leq \frac{a^*(0)}{a_n^*(0)} \| B\|_{A_1} \left\| \frac{A_n^*}{A^*} \right\|_{A_1} \leq \| B\|_{A_1} \left\| \frac{A_n^*}{A^*} \right\|_{A_1}\, ,
$$
where we used that $a^* (0)/a_n ^*(0) \leq 1$, which follows from the fact that $a^* (0) = a_n ^* (0) a_{>n}^* (0)$, where $a_{>n} ^*$ comes from the potential $F \mathbf{1}_{(n, \infty)}$ (see \cite[(7.3)]{QSP_NLFA}).
By \cite[Lem. 6]{SU2Baxter}, as $\|b\|_{A_1}<1$, then $\|A_n/A\|_{A_1} < 1/(1-\|b\|_{A_1}^2)$. Thus,
$$
\|B_n\|_{A_1} \leq \frac{\|B\|_{A_1} }{1-\|b\|_{A_1}^2} = \frac{\|B\|_{A_1} }{1- a^*(0)^2 \|B\|_{A_1}^2} = \frac{\|b\|_{A_1}}{ a^* (0) (1- \|b\|_{A_1 ^2}) } \, ,
$$
which is at most $2- \epsilon$ precisely by assumption. This completes the proof of Theorem \ref{prop:positive with a}. \qed

While we do not flesh out the details in this manuscript, one can also apply a similar argument with the second equation of \eqref{eq:normalized recursion} to obtain an estimate which holds whenever $\|A_n\|_{A_1} < 2- \lil$ instead of $\|B_n\|_{A_1} < 2-\lil$.

\section{Obstructions for further improvements} 

Perhaps the most interesting open question motivated by the results of this paper is Question \ref{quest_theshold}, namely, what is the largest $\tau>0$ with the following property: for any $\epsilon>0$, for any potential $F\in \ell^2$ with $a_F$ outer in $\mathbb D^*$, the inequality $\|b_F\|_{A_1}<\tau-\epsilon$ implies $\|F\|_{\ell^1}\lesssim_{\varepsilon}1$. Theorems \ref{prop:counterexamples a b F l1 A1} and \ref{prop:positive} together imply that $$1/\sqrt{3}\leq \tau\leq 1 \, .$$

One can ask whether the upper bound $1$ above can be improved by an appropriate modification of our construction. Let us first recall how the upper bound $1$ was achieved. We set $F=F'+H$, where $\|H\|_{\ell^1}\approx M$ and $\|H\|_{\ell^{\infty}}=\varepsilon$ with $M$ being arbitrarily large but fixed while $\varepsilon$ small and eventually tending to zero. Using \eqref{eq:prod F' and H} we wrote
\begin{equation}\label{eq:b_F in 4 terms}
    b_F = b_{F'} + a_{F'}\cdot \widehat{H} + a_{F'} (b_H-\widehat{H}) + b_{F'} (a^*_H-1),
\end{equation}
and then showed that the $A_1$ norms of the last three terms were $o(1)$ as $\varepsilon\to 0$ under our choices of parameters. Namely, the last two terms were $o(1)$ due to the inner cancellations inside $H$ by Proposition \ref{prop:counterexamples a b F l1 A1} and $\|a_{F'}\cdot \widehat{H}\|_{A_1}=o(1)$ because of the cancellations coming from $p$-periodicity of $H$ and the choice of $F'$. More precisely, from the multilinear expansion point of view,
\begin{align*}
    a_{F'}\cdot \widehat{H}=\sum_{n=0}^{\infty}(-1)^n\sum_{j_1<...<j_{2n}<0\leq j_{2n+1}} z^{\sum_{k=1}^{2n+1}(-1)^{k+1}j_k}F'_{j_1}...{F'_{j_{2n}}}H_{j_{2n+1}},
\end{align*}
and we had $F'_{-p-1}F'_{-1}-1=o(1)$, which ensured
\begin{align*}
    \sum_{j_1<j_{2}<0\leq j_3}(-1)^{\frac{3-1}{2}}F_{j_1}F_{j_2}F_{j_3}+\sum_{0\leq j_1}(-1)^{\frac{1-1}{2}}F_{j_1}=o(1),
\end{align*}
and thus, the fact that the support of $F'$ belongs to some residue class $\mod p$ together with $p$-periodicity of $H$ yielded $\|a_{F'}\cdot \widehat{H}\|_{A_1}=o(1)$. Finally, we had $\|b_{F'}\|_{A_1} = 1 + o(1)$ by the choice of $F'$.

A natural attempt to generalize this construction and make it more flexible in order to get a better bound for $\|b_F\|_{A_1}$ would be to enrich the ``additional" potential $F'$. More specifically, the question is: can we keep $H$ with the parameters $M,\varepsilon$ as above and add to it another potential $F'$ that is allowed to be supported on $p\Z_{\leq 0}-1$  with $\|F'\|_{\ell^1}$ bounded as $\varepsilon\to 0$ in such a way that $\|b_F-b_{F'}\|_{A_1}=o(1)$ as $\varepsilon\to 0$ and $\|b_{F'}\|_{A_1}<1-\beta$ for some $\beta>0$ independent of $\varepsilon$? Unfortunately, the answer to this question is negative. 

Indeed, assume that we succeeded in constructing such a potential. Then, as before, we have the last two terms in \eqref{eq:b_F in 4 terms} being $o_{\|F'\|_{\ell^1}}(1)$ due to the estimate
$$\|a_{F'} (b_H-\widehat{H})\|_{A_1} \leq \|a_{F'}\|_{A_1} \|b_H-\widehat{H}\|_{A_1} \leq e^{\|F\|_{\ell^1}} \|b_H-\widehat{H}\|_{A_1}$$
where in the last step we used \eqref{eq:norm bd a+b F}. A similar estimate applies for the second term. Next, as $a_{F'}$ has nonzero Fourier coefficients only at $p\Z$, while $\sum_{m=0}^{p-1}H_mz^m$ is supported in $[0,p-1]\cap \Z$, we obtain
\begin{align*}
\| a_{F'} \cdot \widehat{H}\|_{A_1} &= \Big\|\Big(\sum_{m=0}^{p-1}H_mz^m\Big)\Big\|_{A_1}
\Big\|\Big(\sum_{m=0}^{p-1}z^{pm}\Big)a_{F'}(z)\Big\|_{A_1}\\
&\geq (p-1) \varepsilon \Big\|\Big(\sum_{m=0}^{p-1}z^{pm}\Big)a_{F'}(z)\Big\|_{A_1}\\
&\geq p(p-1) \varepsilon |a_{F'}(1)| =\|H\|_{\ell^1} |a_{F'}(1)|.
\end{align*}
Hence, in order to get $\|b_{F}-b_{F'}\|=o(1)$, we need that
\begin{align}\label{eq:a'}
    |a_{F'}(1)|\leq\frac{\| a_{F'} \cdot \widehat{H}\|_{A_1}}{\|H\|_{\ell^1}}=o(1),
\end{align}
which in turn implies $\|b_{F'}\|_{A_1}\geq \|b_{F'}\|_{L^\infty}\geq \sqrt{1-|a_{F'}(1)|^2}\geq 1-o(1)$, which means that $\|b_{F'}\|_{A_1}$ cannot be bounded away from $1$ from below.

Importantly, this argument also shows why such a construction leads to a dead end for the $SU(1,1)$ case. Indeed, we can repeat the same steps for the $SU(1,1)$ case and arrive again at \eqref{eq:a'}. However, for the $SU(1,1)$ case, \eqref{eq:a'} cannot hold if $\varepsilon$ is sufficiently small, since $|a_{F'}(z)|\geq 1$ on $\T$. 

Similarly, if we write the analogous condition that guarantees boundedness of $\|a_{F}\|_{A_1}$ with $\|F\|_{\ell^1}$ being arbitrarily large, we will arrive at $b_{F'}(1)=o(1).$ Hence, having both $\|a_F\|_{A_1}$ and $\|b_F\|_{A_1}$ bounded while $\|F\|_{\ell^1}$ tends to infinity (cf. Question \ref{question:SU2_aplusb}) would imply $|a_{F'}(1)|+|b_{F'}(1)|=o(1)$, which is impossible for $\varepsilon$ small enough both in $SU(2)$ and $SU(1,1)$ settings.

Nevertheless, to give some ground for optimism, we note that in the argument above we were heavily leaning on the triangle inequality $\|b_{F}\|_{A_1}\leq \|b_{F}-b_{F'}\|_{A_1}+\|b_{F'}\|_{A_1}$, which in some cases may destroy some hidden cancellation.

\bibliographystyle{amsalpha}

\bibliography{references}

\end{document}